\documentclass{amsart}                 
\usepackage[T1]{fontenc}                   
\usepackage[utf8]{inputenc}                 
\usepackage[english]{babel}       
\usepackage{graphicx}                       %
\usepackage{verbatim}
\usepackage{color}
\usepackage{picture}
\usepackage{amsmath,amsthm,amsfonts,amssymb}
\usepackage{mathtools}
\usepackage{bigints}
\usepackage{cleveref}

\newtheorem{thm}{Theorem}[section]
\newtheorem{lem}[thm]{Lemma}

\newtheorem{prop}[thm]{Proposition}

\newtheorem{defn}[thm]{Definition}
\theoremstyle{definition}
\newtheorem{rem}[thm]{Remark}
\theoremstyle{remark}

\DeclareMathOperator{\spt}{spt}
\DeclareMathOperator{\Tr}{Tr}

\newcommand{\R}{\mathbb{R}}
\newcommand{\N}{\mathbb{N}}

\newcommand{\Rn}{\mathbb R^n}

\newcommand{\de}{\partial}
\newcommand{\eps}{\varepsilon}

\DeclareMathOperator{\dive}{div}

\DeclareMathOperator{\esssup}{ess\,sup}

\newcommand{\ds}{\displaystyle}
\DeclareMathOperator{\Qp}{\mathcal Q_p}

\title{ON THE SECOND ANISOTROPIC CHEEGER CONSTANT AND RELATED QUESTIONS}
\author{Gianpaolo Piscitelli}
\address{Dipartimento di Scienze Economiche Giuridiche Informatiche e Motorie, Universit\'a degli studi di Napoli Parthenope, Via Guglielmo Pepe, Rione Gescal - 80035 Nola (NA), Italy.}
\email{gianpaolo.piscitelli@uniparthenope.it}

\date{}

\begin{document}

\maketitle

\begin{abstract}
In this paper we study the behavior of the second eigenfunction of the anisotropic $p$-Laplace operator
\[
-\mathcal Q_{p}u:=-\dive \left(F^{p-1}(\nabla u)F_\xi (\nabla u)\right),
\]
as $p \to 1^+$, where $F$ is a suitable smooth norm of $\R^{n}$. Moreover, for any regular set $\Omega$, we define the second anisotropic Cheeger constant as 
\begin{equation*}
h_{2,F}(\Omega):=\inf \left\{ \max\left\{\frac{P_{F}(E_{1})}{|E_{1}|},\frac{P_{F}(E_{2})}{|E_{2}|}\right\},\; E_{1},E_{2}\subset \Omega, E_{1}\cap E_{2}=\emptyset\right\},
\end{equation*}
where $P_{F}(E)$ is the anisotropic perimeter of $E$, and study the connection with the second eigenvalue of the anisotropic $p$-Laplacian. Finally, we study the twisted anisotropic $q$-Cheeger constant with a volume constraint.
\\
 
\end{abstract}

\section{Introduction}
In the recent years, the interest on anisotropic elliptic problems raised in many contexts. In particular, they have been studied properties of the first (\cite{BFK, DPG2, DPG3,DPdBG}) and second (\cite{DPGP2}) Dirichlet eigenvalue of the anisotropic $p$-La\-pla\-cian operator
\begin{equation}
\label{operat-intro}
-\mathcal Q_{p}u:=-\dive \left(F^{p-1}(\nabla u)\nabla_{\xi}F (\nabla u)
\right),
\end{equation}
where $1<p<+\infty$, and $F$ is a sufficiently smooth norm on $\R^{n}$ (see \Cref{anintro} for the precise assumptions on $F$). Let us observe that the operator in \eqref{operat-intro} reduces to the $p$-Laplacian when $F$ is the Euclidean norm on $\R^{n}$ and, in general, $\mathcal Q_{p}$ is anisotropic and can be highly nonlinear. 

We will consider the following Dirichlet eigenvalue problem:
\begin{equation}
\label{introauti}
\left\{
\begin{array}{ll}
-\mathcal Q_{p}u=\lambda_{F}(p,\Omega) |u|^{p-2}u  &\text{in}\ \Omega \\
u=0 &\text{on } \partial\Omega
\end{array}
\right.
\end{equation}
and we study the limiting behaviour of the second eigenvalue $\lambda_{2,F}(p,\Omega)$ as $p\to 1^+$. 

Regarding the Euclidean case, the asymptotic behavior, as $p\to 1^+$, of the solution of $\Delta_p u=f$ has been studied in \cite{CT} for very general $f$. Moreover, we mention that the asymptotic behavior of the first eigenpair of the $p$-Laplacian,  for $p\to 1^+$, has been studied in \cite{MRST} and, for $p\to +\infty$, in \cite{JLM} (in presence of Dirichlet boundary conditions) and in \cite{EKNT} (in presence of Neumann boundary conditions). 

On the other hand, regarding the anisotropic case, the asympthotic behavior of the first Dirichlet eigenvalue $\lambda_{1,F}(p,\Omega)$ of $-\mathcal Q_{p}$ has been studied in \cite{KN}, for $p\to 1^+$, and in \cite{BKJ}, for $p\to+\infty$. For sake of completeness, we remark that the anisotropic $p$-Laplacian eigenvalue problem has been studied in \cite{DPGP1} when Neumann boundary conditions hold (the limiting behavior as $p\to+\infty$ has been studied in \cite{Pi2}) and in \cite{DPG2} when Robin boundary conditions hold (the limiting behavior as $p\to 1^+$ has been studied in \cite{BDPP} by the use of $\Gamma$-convergence results).

The smallest eigenvalue $\lambda_{1,F}(p,\Omega)$ of \eqref{introauti} is variationally characterized as
\begin{equation}
\label{Raylegh}
\lambda_{1,F}(p,\Omega)=\min_{u\in W_0^{1,p}(\Omega)\backslash\{0\} }\frac{\ds\int_\Omega F^p(\nabla u) dx}{\ds\int_\Omega | u|^p dx}
\end{equation}
and the first eigenfunction of \eqref{introauti} is a minimizer of \eqref{Raylegh}.
We show the following Cheeger inequality:
\begin{equation}
\label{Chetype}
\lambda_{1,F}(p,\Omega)\geq\left(\frac{h_{1,F}(\Omega)}{p}\right)^p
\end{equation}
where $h_{1,F}(\Omega)$ is the anisotropic Cheeger constant, defined as
\begin{equation}\label{c1_intro}
h_{1,F}(\Omega):=\inf_{\substack{
{E\subseteq\Omega}
}}\left\{\frac{P_F(E)}{|E|}\right\},
\end{equation}
with $P_F$ denoting the anisotropic perimeter in $\R^n$. 
Let us remark that the first eigenvalue of the anisotropic $1$-Laplacian is defined by
 \begin{equation}
 \label{1Lapanis}
 \lambda_{1,F}(1,\Omega)=\min_{0\neq u\in BV(\Omega)}\frac{|Du|_F(\Omega)+\int_{\partial\Omega}|u|F(\boldsymbol\nu_\Omega)d\mathcal H^{n-1}}{\int_\Omega |u|dx},
 \end{equation}
 where $BV(\Omega)$ denotes the spaces of bounded variation in $\Omega$ and $|Du|_F(\Omega)$ the total variation of $u$ with respect to $F$ and $\boldsymbol\nu_\Omega$ the outer normal to $\partial\Omega$. Problems \eqref{1Lapanis} and \eqref{c1_intro} are equivalent in the sense that a function $u\in BV(\Omega)$ is a minimum of \eqref{1Lapanis} if and only if almost every level set is a Cheeger set;  therefore it holds that $\lambda_{1,F}(1,\Omega)=h_{1,F}(\Omega)$. To study in deep this argument, see \cite{CCMN} and reference therein); refer to \cite{DPNOT,DPOS,FPSS,KF,KS,SS} for the Euclidean case (that is when $F=\mathcal E$) and to \cite{CSZ, CEP} for the discrete case.

In this paper, we show that the first eigenvalue $\lambda_{1,F}(p,\Omega)$ tends to the Cheeger constant:
\begin{equation}
\label{firsttends}
\lim_{p\to 1^+}\lambda_{1,F}(p,\Omega)^{\frac1p}=h_{1,F}(\Omega).
\end{equation}
Let us highlight that both the Cheeger inequality \eqref{Chetype} and the limiting behaviour \eqref{firsttends} have been proven in \cite{BFK}. We overcome some regularity assumptions (see \cite{Li,S}) by using an interior approximation for $BV$ functions by smooth functions with compact support \cite[Theorem 4.3]{BDPP}. 
Thank to this result, we are able to prove that the Cheeger set $C$ can be approximated by a sequence of sets $C_k$ of finite perimeter, compactly contained in $C$ and with smooth boundary.

 The main aims of this paper are the proofs of the analogous of \eqref{Chetype} and \eqref{firsttends} for the  second anisotropic eigenvalue $\lambda_{2,F}(p,\Omega)$ of \eqref{introauti}, as announced in \cite{DPGP2}. More precisely, we prove the following anisotropic Cheeger inequality
\begin{equation}
\label{ci_2_ineq}
\lambda_{2,F}(p,\Omega)\geq\left(\frac{h_{2,F}(\Omega)}{p}\right)^p,
\end{equation}
where $h_{2,F}(\Omega)$ denotes the second anisotropic Cheeger constant
\begin{equation*}
h_{2,F}(\Omega):=\inf \left\{ \max\left\{\frac{P_{F}(E_{1})}{|E_{1}|},\frac{P_{F}(E_{2})}{|E_{2}|}\right\},\;|E_{1}|,|E_{2}|>0,\; E_{1},E_{2}\subset \Omega, E_{1}\cap E_{2}=\emptyset\right\},
\end{equation*}
and the following limiting relationship 
\begin{equation}\label{2limiting}
\lim_{p\to 1^+}\lambda_{2,F}(p,\Omega)^{\frac1p}=h_{2,F}(\Omega).
\end{equation}
Let us remark that, for the Euclidean case, both the Cheeger inequality \eqref{ci_2_ineq} and the limit property \eqref{2limiting} have been proven in \cite{Pa1,BoP}.

The asymptotic result can be easily generalized to the $k$-th variational eigenvalue obtaining
\begin{equation}\label{3limiting}
\limsup_{p\to 1^+}\lambda_{k,F}(p,\Omega)^\frac 1p\leq h_{k,F}(\Omega)\quad\text{for}\quad k\geq 3,
\end{equation}
where $h_{k,F}(\Omega)$ is the $k$-th Cheeger constant. 
The deep reason of the discrepancy between \eqref{2limiting} and \eqref{3limiting} rely on the fact that the $k-$th eigenfunction usually does not have $k-$th nodal domains. 
 
Finally, we observe that the second anisotropic Cheeger constant is related to the anisotropic ``twisted'' $q$-Cheeger constant, $1\le q<\frac{n}{n-1}$:
\begin{equation*}
\mathcal{K}_{q,F}(\Omega):=\inf\left\{\frac{ |Du|_F(\R^n)}{\ds\left(\int_\Omega|u|^{q} dx\right)^{\frac 1q}};\,u\in BV_0(\Omega), \ u\not\equiv 0, \ \int_\Omega u \ dx =0\right\},
\end{equation*}
where $BV_{0}(\Omega)$ contains the $BV$ functions in $\R^{n}$ which vanish outside $\Omega$. Moreover, it holds that
\[
h_{1,F}(\Omega)\le\mathcal{K}_{1,F}(\Omega)\le h_{2,F}(\Omega).
\]
In the Euclidean case, the minimization problem of $\mathcal K_{1,F}(\Omega)$ under volume constraint has been studied in \cite{BDNT}. 

The paper is organized as follows. In \Cref{anintro}, we recall and fix some definitions and results of the Finsler geometry; in \Cref{known_facts_sec}, we study the eigenvalue and the Cheeger problems; in \Cref{asymptotic_sec}, we give the two main results; in \Cref{twisted_sec}, we study the anisotropic twisted Cheeger constant.

\section{Notations and preliminaries}\label{anintro}
In this section, we fix the notations of the paper and recall some useful results of the Finsler geometry.
\subsection{The norm}
Throughout the paper we will consider a $C^2(\R^n\backslash\{0\})$ convex even $1-$homogeneous function 
\[
\xi\in \R^{n}\mapsto F(\xi)\in [0,+\infty[,
\] 
that is a convex function such that
\begin{equation}
\label{eq:omo}
F(t\xi)=|t|F(\xi), \quad t\in \R,\,\xi \in \R^{n}, 
\end{equation}
 and
\begin{equation}
\label{eq:lin}
a|\xi| \le F(\xi),\quad \xi \in \R^{n},
\end{equation}
for some constant $a>0$. 
Under this hypothesis, it is easy to see that there also exists $b\ge a$ such that
\begin{equation}
\label{finbi}
F(\xi)\le b |\xi|,\quad \xi \in \R^{n}.
\end{equation}
Moreover, throughout this paper, we will assume that 
\begin{equation}
\label{strong}
\nabla^{2}_{\xi}[F^{p}](\xi)\text{ is positive definite in }\R^{n}\setminus\{0\},
\end{equation}
where $1<p<+\infty$. 

The hypothesis \eqref{strong} on $F$ assures that the operator 
\[
\Qp u:= \dive \left(\frac{1}{p}\nabla_{\xi}[F^{p}](\nabla u)\right)
\] 
is elliptic, hence there exists a positive constant $\gamma$ such that
\begin{equation*}
\frac1p\sum_{i,j=1}^{n}{\nabla^{2}_{\xi_{i}\xi_{j}}[F^{p}](\eta)
  \xi_i\xi_j}\ge
\gamma |\eta|^{p-2} |\xi|^2, 
\end{equation*}
for some positive constant $\gamma$, for any $\eta \in
\Rn\setminus\{0\}$ and for any $\xi\in \Rn$. 

The polar function $F^o\colon\R^n \rightarrow [0,+\infty[$ 
of $F$ is defined as
\begin{equation*}
F^o(v)=\sup_{\xi \ne 0} \frac{\langle \xi, v\rangle}{F(\xi)}. 
\end{equation*}
 It is easy to verify that also $F^o$ is a convex function
which satisfies properties \eqref{eq:omo} and
\eqref{eq:lin}. Furthermore, we have
\begin{equation*}
F(v)=\sup_{\xi \ne 0} \frac{\langle \xi, v\rangle}{F^o(\xi)}.
\end{equation*}
From the above property it holds that
\begin{equation*}
|\langle \xi, \eta\rangle| \le F(\xi) F^{o}(\eta), \qquad \forall \xi, \eta \in \R^{n}.
\end{equation*}
Furthermore, the following properties of $F$ and $F^o$ hold true (see for example \cite{AB,AFLT,BP}):
\begin{gather}
 \nabla_{\xi}F(\xi) \cdot \xi = F(\xi), \quad  \nabla_{\xi}F^{o} (\xi) \cdot \xi = F^{o}(\xi)\qquad\quad\ \ \forall\xi\in\R^n\setminus \{0\}, \label{eq:om}\\
 F( \nabla_{\xi}F^o(\xi))=F^o( \nabla_{\xi}F(\xi))=1\qquad\qquad\qquad\qquad  \forall \xi \in\R^n\setminus \{0\}, \label{eq:H1} \\
F^o(\xi)  \nabla_{\xi}F(\nabla_{\xi}F^o(\xi) ) = F(\xi) \nabla_{\xi}F^o( \nabla_{\xi}F(\xi) ) = \xi\quad \forall \xi \in\R^n\setminus \{0\}, \label{eq:HH0}\\
\label{eq:sec}\sum_{j=1}^n\nabla^2_{\xi_i\xi_j}F(\xi)\xi_i\xi_j=0 \quad \forall\ i=1,...,n \qquad\qquad\qquad \forall \xi \in\R^n\setminus \{0\}.
\end{gather}

The anisotropic distance function to $\partial \Omega$ is
\[d_F(x)=\inf_{y\in\partial \Omega}F^o(x-y)\quad x\in\overline\Omega,\]
and solves the anisotropic iconal equation $F(\nabla d_F(x))=1$.
The set
\[
\mathcal W = \{  \xi \in \R^n \colon F^o(\xi)< 1 \}
\]
is the so-called Wulff shape centered at the origin. We put
$\kappa_n=|\mathcal W|$, where $|\mathcal W|$ denotes the Lebesgue measure
of $\mathcal W$. More generally, we denote with $\mathcal W_r(x_0)$
the set $r\mathcal W+x_0$, that is the Wulff shape centered at $x_0$
with measure $\kappa_nr^n$, and $\mathcal W_r(0)=\mathcal W_r$.

\subsection{The Hausdorff measure and the perimeter} 
In this section we recall some useful notion related to the perimeter in the Finsler metric (for more details we refer to \cite{AB,BP,M}). 

For any measurable subset $E$ of $\Omega$, the ipersurface element induced by $F$ is
\begin{equation}
\label{FHausdorff}
d\mathcal H^{n-1}_F(E):=F(\boldsymbol\nu_E) \ d\mathcal{H}^{n-1}(E),
\end{equation}
where $\boldsymbol\nu_E$ is the Euclidean outer normal to $E$ and $d\mathcal {H}^{n-1}$ is the $n-1$-dimensional Hausdorff measure. Moreover, we will denote $\boldsymbol n_E:=\nabla F(\boldsymbol\nu_E)$ the anisotropic normal to $\de E$.

We observe that the Carath\'eodory's construction \cite[Chap. 2.10]{F} and the definition of the Hausdorff measure with respect to the Finsler metric $F$ \eqref{FHausdorff}, lead to the definition
\begin{equation*}
\mathcal H^{n-1}_{F,\delta}(E):=\inf \left\{\ds\sum_{j\in J}
\kappa_{n-1}
r_j^{n-1}
\ : \ 
r_j<\delta,\ E\subset \bigcup_{j\in J} \mathcal W_{r_j}(x_j) \right\},
\end{equation*}
where $J
$ is a finite or countable set of indexes
. Therefore, it holds
\begin{equation*}
\mathcal H^{n-1}_F (E)=\lim_{\delta\to 0^+} \mathcal H^{n-1}_{F,\delta} (E)=\sup_{\delta > 0} 
\mathcal H^{n-1}_{F,\delta}(E)
.
\end{equation*}

Let $u\in BV(\Omega)$, the \emph{total variation of $u$ with respect to $F$} is
\[
|D u|_F (\Omega)
=\sup\left\{ \int_\Omega u \dive\sigma\ dx \ : \ \sigma\in C^1_0(\Omega ; \R^n), F^o(\sigma)\leq 1 \right\}
\]
and the \emph{perimeter of a set $E$ with respect to $F$} is:
\[
P_F(E;\Omega)=|D \chi_E|_F (\Omega)
=\sup\left\{ \int_E \dive\sigma\ dx \ : \ \sigma\in C^1_0(\Omega ; \R^n), F^o(\sigma)\leq 1 \right\}.
\]
Moreover, if $E$ has Lipschitz boundary, then it holds
\[
P_F(E;\Omega)=\int_{\Omega\cap\partial E} F(\boldsymbol\nu_E) \ d\mathcal{H}^{n-1}.
\]

Throughout this paper ,we denote $P_F(E)$ instead of $P_F(E;\R^{n})$ for any set $E\subset\Omega$ with finite perimeter.
Let us remark that, for all subsets $E$ of $\R^n$, $P_F(E;\Omega)$ is finite if and only if the usual perimeter $P(E;\Omega)$ is finite; indeed, by \eqref{eq:lin} and \eqref{finbi}, we have 
\[
a P(E;\Omega)\leq P_F(E;\Omega)\le b  P(E;\Omega).
\]

Moreover, a coarea formula holds (\cite{BBF,CFM,ET,FV}):
\begin{equation}\label{Fcoarea}
|D u|_F (\Omega)
=\int_{\R}P_F(\{u>s\};\Omega)ds
\end{equation}and, for all $u\in W^{1,1}(\Omega)$, we have 
\begin{equation}\label{bvrel}
|D u|_F (\Omega)
=\int_\Omega F(\nabla u)\ dx.
\end{equation} 
By using \eqref{Fcoarea} and \eqref{bvrel} with $u=d_F$, we have
\begin{equation}\label{acv}
|\Omega |=  \int_\Omega \ dx =\int_\Omega F(\nabla d_F(x))\ dx =  |Dd_F|_F(\Omega)=\int_{\R}P_F(\{d_F>s\};\Omega)\ ds.
\end{equation}

Furthermore, by using the divergence theorem and \eqref{eq:om}-\eqref{eq:H1}, we have
\[
\begin{split}
P_F(\mathcal W)&=\int_{\partial \mathcal W_r}F(\nu)d\mathcal H^{n-1}(x)=\int_{\partial \mathcal W_r}F\left(\frac{\nabla F^o(x)}{|\nabla F^o(x)|}\right)d\mathcal H^{n-1}(x)\\
&=\int_{\partial \mathcal W_r}\frac{1}{|\nabla F^o(x)|}d\mathcal H^{n-1}(x)=\int_{\partial \mathcal W_r}\frac{F(x)}{|\nabla F^o(x)|}d\mathcal H^{n-1}(x)\\
&=\frac 1r\int_{\partial \mathcal W_r}\frac{x\cdot \nabla F(x)}{|\nabla F^o(x)|}d\mathcal H^{n-1}(x)=\frac 1r\int_{\partial \mathcal W_r}x\cdot \nu d\mathcal H^{n-1}(x)=\frac nr\int_{\mathcal W_r}dx=nk_nr^{n-1}.
\end{split}
\]

The anisotropic isoperimetric inequality \cite{B, AFLT} states that among set $E$ of finite perimeter in $\R^n$ and of fixed volume the Wulff shape minimizes the anisotropic perimeter, that is:
\begin{equation}
\label{isoine}
P_F(E;\R^n)\geq n \kappa_n^\frac 1n |E|^{1-1/n}.
\end{equation}
The equality in \eqref{isoine} holds if and only if $E=\mathcal W_{R}$.

We recall from \cite[Theorem 4.3]{BDPP} a result of approximation of $\Omega$ strictly from within, that is a key result for our aims.
\begin{prop}
Let $\Omega\subset\R^N$ be an open bounded set with Lipschitz boundary  and let $u\in BV(\Omega)\cap L^p(\Omega)$ for some $p\in [1,\infty)$. Then there exists a sequence $\{u_k\}_{k\in\mathbb N}\subseteq C_0^\infty(\Omega)$ such that, for any $q\in [1,p]$,
\[
u_k\to u\ \ \text{in}\ \ L^q(\Omega)\quad\text{and}\quad |D u_k|_F(\mathbb R^N)\to |D u|_F(\mathbb R^N).
\]
\end{prop}
In the sequel, the previous result will be useful when particularized to characteristic functions of sets of finite perimeter.
\begin{prop}
\label{approssimazione_thm_per}
Let $\Omega\subset\R^N$ be an open bounded set with Lipschitz boundary with finite perimeter. Then there exists a $C_0^\infty$ sequence $\{E_k\}_{k\in\mathbb N}\subseteq \Omega$ such that
\[
P_F(E_k)\to P_F(\Omega).
\]
\end{prop}

\section{The eigenvalue and the Cheeger problem}
\label{known_facts_sec}
 In this Section, we firstly recall some results on the first anisotropic $p$-Laplacian eigenfunction(s) and eigenvalue, on the first and second anisotropic Cheeger set(s) and constant; then we discuss the relationships among them, as $p\to 1^+$.
 
\subsection{The Dirichlet anisotropic $p$-Laplace eigenvalue problem}
We first recall the definition of eigenvalue of the Dirichlet anisotropic $p$-Laplace operator \eqref{operat-intro}.
\begin{defn}
We say that $u\in W_{0}^{1,p}(\Omega)$, $u\not\equiv 0$, is an eigenfunction of \eqref{introauti}, if 
\[
\int_\Omega F^{p-1}(\nabla u) \nabla_\xi F(\nabla u)\cdot\nabla \varphi \ dx=\lambda_F(p,\Omega)\int_\Omega |u|^{p-2}u\varphi\ dx
\]
for all $\varphi\in W^{1,p}_0(\Omega)$. The corresponding real number $\lambda_F(p,\Omega)$ is called an eigenvalue of \eqref{introauti}.
\end{defn}

The following Faber-Krahn type inequality has been proved in \cite[Theorem 3.3]{BFK}. 
\begin{prop}
Let $\Omega$ be a bounded connected open set. The first eigenvalue of the Dirichlet problem \eqref{introauti} is simple and the first eigenfunction is positive. Moreover
\begin{equation}
\label{FabKra}
|\Omega|^{p/n} \lambda_{1,F}(p,\Omega)\geq\kappa_n^{p/n}\lambda_{1,F}(p,\mathcal W_R),\quad\text{with}\ |\mathcal W_R|=|\Omega|
\end{equation}
and the equality sign in \eqref{FabKra} holds if $\Omega$ is homothetic to the Wulff shape.
\end{prop}

Many definitions hold for the second eigenvalues, mainly based on the Liusternik-Schnirelmann theory applied to the study of the critical points of the Raylegh quotient. 
As examples, the min-max formulas relying on topological index theories (see \cite{Ra}) involving the notion of Krasnoselskii genus, or the class of all odd and continuous functions mapping the spherical shell into suitable normalized functional space (see e.g. Sections 3-4 in \cite{DPGP2} and reference therein to more investigate the issue).

Particularly, we will use the following definition of the second eigenvalue of problem \eqref{introauti}.
\begin{defn}
Let $\Omega$ be a bounded open set of $\R^{n}$. Then the second eigenvalue of \eqref{introauti} is
\begin{equation*}
\lambda_{2,F}(p,\Omega):=\begin{cases}
\min\{\lambda > \lambda_{1,F}(p,\Omega)\colon \lambda\ \text{is an eigenvalue}\}& \text{if } \lambda_{1,F}(p,\Omega) \text{ is simple},\\
\lambda_{1,F}(p,\Omega) & \text{otherwise}.
\end{cases}
\end{equation*}
\end{defn}

In the sequel, we will prove a Cheeger type inequality for the second eigenvalue. To this aim, we recall the following key result from \cite{DPGP2}.
\begin{prop}
\label{key}
Let $\Omega$ be an open bounded set of $\R^n$, then for every $1<p<+\infty$, there exist two disjoint domains $\Omega_1$, $\Omega_2$ of $\Omega$ such that
\[
\lambda_{2,F}(p,\Omega)=\max \{\lambda_{1,F}(p,\Omega_1), \lambda_{1,F}(p,\Omega_2)\}.
\]
\end{prop}

In order to prove the main result of this paper, we need the following characterization of $\lambda_{2,F}$.
\begin{prop}
\label{l2}
Let $\Omega$ be a bounded domain of $\R^{n}$. Then
\[
\lambda_{2,F}(\Omega)=\ell_{2,F}(\Omega):=\inf\left\{\max\left\{\lambda_{1,F}(p,\Omega_1),\lambda_{1,F}(p,\Omega_2)\right\},\; \Omega_{1},\Omega_2\subseteq \Omega,\;\Omega_{1}\cap\Omega_{2}=\emptyset\right\}.
\] 
\end{prop}
\begin{proof}
From the Proposition \ref{key}, we get immediately that
\[
\lambda_{2,F}(\Omega)\ge \ell_{2,F}(\Omega).
\]
On the other hand, for any $\eps>0$ there exist two disjoint sets $\bar\Omega_{1},\bar\Omega_{2}\subset \Omega$, such that
\[
\ell_{2,F}(\Omega)+\eps> \max\{\lambda_{1,F}(\bar\Omega_{1}),\lambda_{1,F}(\bar\Omega_{2})\} \ge \lambda_{2,F}(\bar\Omega_{1}\cup\bar\Omega_{2})\ge \lambda_{2,F}(\Omega).
\]
Being $\eps>0$ arbitrary, we get
\[
\ell_{2,F}(\Omega) \ge \lambda_{2,F}(\Omega).
\]
\end{proof}

Now we state the Hong-Krahn-Szego inequality for $\lambda_{2,F}(p,\Omega)$. 
\begin{prop} 
Let $\Omega$ be an open bounded set of $\R^{n}$. Then\begin{equation}\label{HKS}\lambda_{2,F}(p,\Omega)\geq  \lambda_{2,F}(p,\widetilde{\mathcal W}),\end{equation}where $\widetilde{\mathcal W}$ is the union of two disjoint Wulff shapes, each one of measure $\frac{|\Omega|}{2}$. Moreover equality sign in \eqref{HKS} occurs  if  $\Omega$ is the disjoint union of two  Wulff shapes of the same measure.
\end{prop}

\subsection{The Cheeger problem}
We firstly recall the definitions of the anisotropic Cheeger constant and anisotropic Cheeger set. For a survey on the Cheeger problem, we refer to \cite{L,Pa2} and references therein.
\begin{defn}
The anisotropic Cheeger constant is
\begin{equation}
\label{first_cheeger}
h_{1,F}(\Omega):=\inf_{\substack{E\subseteq\Omega}}\frac{P_F(E)}{|E|}.
\end{equation}
Moreover, a set $C\subseteq\Omega$ achieving the minimum in \eqref{first_cheeger} is called an anisotropic Cheeger set for $\Omega$.
\end{defn}
The problem \eqref{first_cheeger} has a unique solution when $\Omega$ has Lipschitz boundary (see \cite[Corollary 6.7]{CFM}).

Thank to \Cref{approssimazione_thm_per}, we are able to state that the minimum in the anisotropic Cheeger constant \eqref{first_cheeger} is achieved in the class of smooth sets well contained in $\Omega$.
\begin{prop} Let $\Omega\subset\R^N$ be an open bounded set with Lipschitz boundary with finite perimeter.
It holds that
\[
h_{1,F}(\Omega)=\inf_{\substack{E\subset\subset\Omega \\[.1cm] \de E \,\text{smooth}}}\frac{P_F(E)}{|E|}.
\]
\end{prop}

From the isoperimetric inequality \eqref{isoine}, it easily holds that
\begin{equation}
\label{faberkrahncheeger}
h_{1,F}(\Omega)\ge h_{1,F}(\mathcal W_R)=
\frac{n \kappa_{n}^{\frac1n}}{|\Omega|^{\frac1n}}.
\end{equation}
where $|\Omega|=|\mathcal W_R|$.

At his stage, we focus on some regularity properties of the anisotropic Cheeger set.

\begin{defn}
\label{weakcurvature}
If $E$ is set with Lipschitz boundary, $E$ has (locally summable) distributional anisotropic mean curvature in $\Omega$, if there exists $H_F\in L^1_{\textrm{loc}}(\Omega\cap\partial E;\mathcal H^{n-1})$ such that
\[
\int_{\partial E}\dive_{F,\de E} g\,F(\boldsymbol\nu_E)d\mathcal H^{n-1}=\int_{\partial E}H_F(\boldsymbol\nu_E \cdot g)d\mathcal H^{n-1}, \quad \forall \ g\in C_c^\infty(\Omega;\R^n).
\]
where
\[
\dive_{F,\de E}g=\Tr\left(\left(Id-\boldsymbol n_{E}\otimes \frac{\boldsymbol\nu_E}{F(\boldsymbol\nu_E)}\right)\nabla g\right).
\]
\end{defn}
We remark that when the boundary of $E$ is of class $C^1$, then the distributional definition of anisotropic mean curvature coincides with the classical definition.

In the next Proposition, we list the main properties of the Cheeger sets and of the Cheeger constant.
\begin{prop}\label{elenco_prop_che}
The following properties hold true. Let $C$ be an anisotropic Cheeger set of $\Omega$.
\begin{enumerate}
\item 
Let $\Omega$ be an open bounded set. Then $\de C\cap \Omega$ is $C^{1,1}$.
\item Let $\Omega$ be an open bounded set.
The anisotropic mean curvature of $\partial C \cap\Omega$ is equal, up to a set of $\mathcal H^{n-1}$ measure zero, to $h_{1,F}(\Omega)$.
\item If $C$ is an anisotropic Cheeger set for $\Omega$, then $\partial C\cap\partial\Omega\not = \emptyset$.
\item 
If $\Omega$ is $C^{1,1}$ and convex, then there is a unique convex Cheeger set.
\end{enumerate}
\end{prop}
\begin{proof}
The proofs of (1) and (4) are contained in \cite[Thms. 6.18-6.19]{ANP} and in (\cite[Thm. 6.3]{CCMN}). 
\\
In order to prove (2), let us observe that $C$ is an anisotropic Cheeger set of $\Omega$ if and only if $|C|>0$ and $C$ minimizes
\begin{equation}
\label{calibrable}
\min_{C\subseteq\Omega} \left[P_F(C)-h_{1,F}(\Omega)|C|\right].
\end{equation}
 Let $\phi_t :\R^n\to\overline{\Omega}$ be a family of diffeomorphisms for $t\in\R$, such that $\phi=Id$ and $\phi_t - Id$ have compact support in $\R^n$. We set $C_t:=\phi_t( C)$.
\\
Then $C$ is a minimum for problem \eqref{calibrable} and hence
\[
\frac{d}{dt}\Bigl(P_F(C_t)-h_{1,F}(\Omega)|C_t|\Bigr)_{|t=0}=0.
\]
\\
By applying area formula as in \cite[Thm 3.6]{BNR}, we have that
\[
\int_{\partial C}\left[\dive g - \left(\nabla g\, \frac{\nu}{F(\nu)}\right)\cdot n_F \right]F(\nu) d\mathcal H^{n-1}-h_{1,F}(\Omega)\int_{\partial C}\left(
\nu\cdot g\right)  d \mathcal H^{n-1}=0,
\]
where $g=\left(\frac{d\phi_t}{dt}\right)_{|t=0}$. Hence by definition \eqref{weakcurvature} we have that
\[
\int_{\partial C}\Big(H_F-h_{1,F}(\Omega)\Big)\nu \cdot g \ d\mathcal H^{n-1} =0
\]
and therefore, by the arbitrariness of $\phi$, it holds that $H_F-h_{1,F}(\Omega)=0$ $\mathcal H^{n-1}$-a.e. on $\de C$. 
\\
Finally, as regards (3), by contradiction we suppose that $\overline C\subset\Omega$. This implies that there exists a real number $\mu>1$ such that the set $\mu C=\{\mu x \ |\ x\in C \}$ is contained in $\Omega$. Hence we have:
\[
\frac{P_F(\mu C)}{|\mu C|}=\frac 1 \mu \frac{P_F(C)}{|C|}<\frac{P_F(C)}{|C|}
\]
that contradicts the minimality of $C$.
This concludes the proof.
\end{proof}

\begin{rem} 
The second item of Proposition \eqref{elenco_prop_che} states that the first anisotropic Cheeger constant $h_{1,F}(\Omega)$ coincides (up to a set of measure $\mathcal H^{n-1}$ zero) with the anisotropic mean curvature of the boundary of the Cheeger set inside $\Omega$. We remark that the authors in \cite{DPG1} proved also that $\partial E\cap\Omega$,  in the plane, is either homothetic to an arc of a Wulff shape or a straight segment.
\end{rem}

\subsection{The second anisotropic Cheeger constant}
We first recall the definition of second anisotropic Cheeger constant.

\begin{defn}
The second anisotropic Cheeger constant is
\begin{equation*}
h_{2,F}(\Omega):=\inf \left\{ \max\left\{\frac{P_{F}(E_{1})}{|E_{1}|},\frac{P_{F}(E_{2})}{|E_{2}|}\right\},\;|E_{1}|,|E_{2}|>0,\; E_{1},E_{2}\subset \Omega, E_{1}\cap E_{2}=\emptyset\right\} .
\end{equation*}
The couples of sets $C_1,C_2$ realizing the min-max of $h_{2,F}$ are called pairs of coupled anisotropic Cheeger sets.
\end{defn}
Equivalently, $h_{2,F}$ can be written as
\begin{equation*}
h_{2,F}(\Omega)=\inf\left\{ \max\left\{h_{1,F}(E_{1}),h_{1,F}(E_{2})\right\},\;|E_{1}|,|E_{2}|>0,\; E_{1},E_{2}\subset \Omega, E_{1}\cap E_{2}=\emptyset\right\}.
\end{equation*}
Refer to \cite[Prop. 3.5]{Pa1} for the proof of this equivalence in the Euclidean case.

In the next result we prove that actually $h_{2,F}(\Omega)$ admits a minimum.
\begin{prop}
\label{second_char1}
There exist two disjoint connected subsets $C_1$ and $C_2$ contained in $\Omega$ such that
\begin{equation}
\label{scc}
h_{2,F}(\Omega)= \max\left\{\frac{P_F(C_1)}{|C_1|},\ \frac{P_F(C_2)}{|C_2|}\right\}.
\end{equation}
\end{prop}
\begin{proof}
Let us consider a minimizing sequence of disjoint couples of sets $(E_{1,n},E_{2,n})_{n\in\N}$, that is
\[
\lim_{n\to+\infty}\max\left\{\frac{P_F(E_{1,n})}{|E_{1,n}|},\ \frac{P_F(E_{2,n})}{|E_{2,n}|}\right\}=h_{2,F}(\Omega).
\]
Then, let us fix $R>0$ the radius of two equal disjoint arbitrary balls contained in $\Omega$; it is possible to consider only the sets such that $|E_{i,n}|\ge |B_R|$, for $i=1,2$ and every $n\in\N$. Indeed, if there exist $\hat i$ and $\hat n$ such that $|E_{\hat i,\hat n}|< |B_R|$, then by \eqref{faberkrahncheeger} we have
\[
\frac{P_F(E_{\hat i,\hat n})}{|E_{\hat i,\hat n}|}\geq h_{1,F}(E_{\hat i,\hat n})\geq h_{1,F}(B_r)> h_{1,F}(B_R),
\]
where $|E_{\hat i,\hat n}|= |B_r|$, with $r<R$. This means that the couple $(E_{1,\hat n},E_{2,\hat n})$ can be discarded.

The compact embedding of $BV(\Omega)$ in $L^1(\Omega)$ assure the existence of a couple  $(C_{1},C_{2})$ such that $\chi_{E_{i,n}}\to\chi_{C_i}$ as $n\to+\infty$ and $|C_{i}|\ge |B_R|$, $i=1,2$. Therefore, by the lower semicontinuity of the anisotropic perimeter, we have
\[
\max\left\{\frac{P_F(C_{1})}{|C_{1}|},\ \frac{P_F(C_{2})}{|C_{2}|}\right\}\leq h_{2,F}(\Omega).
\]
The conclusion follows by showing that $C_1$ and $C_2$ are disjoint. If $x\in C_1$, then $\chi_{C_1}(x)=1$, that means that $\chi_{E_1,n}(x)=1$ definitely. This implies that $\chi_{E_2,n}(x)=0$ definitely, hence that $\chi_{C_2}(x)=0$, and this means that $x\in C_2$. 
\end{proof}

As a consequence of \Cref{second_char1}, we obtain the following lower bounds for $h_{2,F}(\Omega)$.
\begin{prop}\label{low_bou_h2}
Let $\Omega$ be an open bounded domain of $\R^n$. Then:
\begin{enumerate}
\item[(a)] it holds that
\[
h_{2,F}(\Omega)\ge h_{1,F}(\Omega),
\]
and the inequality is strict when $\Omega$ admits a unique anisotropic Cheeger set;
\item[(b)] it holds that
\[
h_{2,F}(\Omega)\ge n \left(\frac{2\kappa_n}{|\Omega|}\right)^\frac 1 n,
\]
\item[(c)] it holds that
\[
h_{2,F}(\Omega)\geq h_{2,F}(\widetilde{\mathcal W}).
\]
where $\widetilde{\mathcal W}$ is the union of two disjoint Wulff shapes, each one of measure $|\Omega|/2$. 
\end{enumerate}
\end{prop}
\begin{proof}
By \eqref{scc} of \Cref{second_char1} and the definition \eqref{first_cheeger} of $h_{1,F}(\Omega)$, the inequality in $\textit{(a)}$ holds.
As regards the strict inequality, if we suppose by contradiction that $h_{1,F}(\Omega)= h_{2,F}(\Omega)$, then there exists a pair of coupled anisotropic Cheeger sets $(C_{1},C_{2})$ such that
\[
 h_{1,F}(\Omega)=\max\left\{\frac{P_F(C_1)}{|C_1|},\frac{P_F(C_2)}{|C_2|}\right\}.
\]
This implies that $h_{1,F}(\Omega)=\frac{P_F(C_1)}{|C_1|}=\frac{P_F(C_2)}{|C_2|}$, contradicting the uniqueness of the anisotropic Cheeger set.
\\
In order to prove \textit{(b)}, let $C_1$ and $C_2$ be a pair of coupled anisotropic Cheeger sets. By \Cref{second_char1} and the isoperimetric inequality \eqref{faberkrahncheeger}, we have 
\begin{multline*}
h_{2,F}(\Omega)\ge \max\{h_{1,F}( C_1),h_{1,F}( C_2)\} \ge \\
\ge \max\{h_{1,F}(\mathcal W_{r_{1}}),h_{1,F}(\mathcal W_{r_{2}})\} = h_{1,F}(\mathcal W_{r_1})\ge n \left(\frac{2\kappa_n}{|\Omega|}\right)^\frac 1 n,
\end{multline*}
where $\mathcal W_{r_{1}}$ and $\mathcal W_{r_{2}}$ are Wulff shape of radii $r_{1}\le r_{2}$ and $|\mathcal W_{r_{i}}|=|C_{i}|$, $i=1,2$. Let us observe that we have $|\mathcal W_{r_{1}}|\le \frac{|\Omega|}{2}$, otherwise $|C_{1}|+|C_{2}|>|\Omega|$, and this is not possible.

To prove (c), let us observe that $n \left(\frac{2\kappa_n}{|\Omega|}\right)^\frac 1 n=h_{1,F}(\widetilde{\mathcal W})=h_{2,F}(\widetilde{\mathcal W})$, where $\widetilde{\mathcal W}$ is the union of two disjoint Wulff shapes, each one of measure $|\Omega|/2$. 
Therefore, \Cref{low_bou_h2} gives the Hong-Krahn-Szego inequality.
\end{proof}

Let us observe that \Cref{low_bou_h2}(b) also holds for higher anisotropic Cheeger constants:
\[
h_{k,F}(\Omega)\geq n \left(k\frac{\kappa_n}{|\Omega|}\right)^\frac 1 n,
\]
where
\begin{equation*}
\begin{split}
h_{k,F}(\Omega):=&\inf \left\{ \max\left\{\frac{P_{F}(E_{k})}{|E_{k}|}, ..., \frac{P_{F}(E_{k})}{|E_{k}|}\right\},\;  |E_{1}|, ..., |E_{k}|>0,\; \right. \\ 
&\qquad\qquad\quad E_{1}, ..., E_{k}\subset \Omega, E_{i}\cap E_{j}=\emptyset\ \forall i \neq j \in \{1, ..., k\}\bigg\} .
\end{split}
\end{equation*}

The Cheeger constants, higher than the first one, are not expected to be regular enough to study the relationship with the asymptotic limit of eigenvalues of the $p$-Laplacian. Therefore, as observed in \cite{BoP}, it is necessary  to require additional conditions to the sets achieving the anisotropic Cheeger constants.
\begin{defn}
Let $\Omega\subset\R^N$ be a measurable set with positive measure.
A couple of sets $(E_1,E_2)$ such that $|E_{1}|,|E_{2}|>0,\; E_{1},E_{2}\subset \Omega, E_{1}\cap E_{2}=\emptyset$ and minimize $h_{2,F}(\Omega)$, is called a $1$-adjusted anisotropic Cheeger couple if
 \begin{equation*}
 \begin{split}
 h_{1,F}(\Omega\setminus E_2)=h_{1,F}(E_1)= & \frac{P_F(E_1)}{|E_1|},\\
 h_{1,F}(\Omega\setminus E_1)=h_{1,F}(E_2)= & \frac{P_F(E_2)}{|E_2|}
 .
 \end{split}
 \end{equation*}
\end{defn}
The existence of $1$-adjusted anisotropic Cheeger couple is proved in the following.
\begin{prop}\label{existence_adj}
Let $\Omega\subset\R^N$ be a bounded measurable set with positive measure.
There exists a $1$-adjusted anisotropic Cheeger couple of $\Omega$.
\end{prop}
\begin{proof}
Let $(C_1,C_2)$ be a couple minimizing $h_{2,F}(\Omega)$ and, without loss of generality, let us assume $\frac{P_F(C_1)}{|C_1|}=h_{2,F}(\Omega)$ and $\frac{P_F(C_2)}{|C_2|}<h_{2,F}(\Omega)$. 

Then a Cheeger set $C'_2$ corresponding to $h_{1,F}(\Omega\setminus C_1)$ is such that $\frac{P_F(C_2')}{|C_2'|}\leq\frac{P_F(C_2)}{|C_2|}<h_{2,F}(\Omega)$. 

If $C_1$ is a Cheeger set for $\Omega\setminus C_2'$, we consider the couple $(C_1,C'_2)$; otherwise, let $C'_1$ be a Cheeger set of $\Omega\setminus C_2'$, and we consider the couple $(C'_1,C'_2)$.
\end{proof}

\section{The asymptotic behaviors}
\label{asymptotic_sec}
In this section, we study the limiting relationships between the eigenvalues and the Cheeger constants.
\subsection{The relationships between $\lambda_{1,F}(p,\Omega)$ and $h_{1,F}(\Omega)$}
From \cite{BFK,KN}, we recall the following Cheeger inequality.
\begin{thm}
Let $\Omega$ be an open bounded subset of $\R^n$ and $1<p<+\infty$, the eigenvalue $\lambda_{1,F}(p,\Omega)$ can be estimated from below as follows:
\begin{equation}\label{cheeger_ine}
\lambda_{1,F}(p,\Omega)\geq\left(\frac{h_{1,F}(\Omega)}{p}\right)^p.
\end{equation}
\end{thm}
\begin{proof}
The proof is adapted from \cite{LW}. 
Let us consider $v \in W_0^{1,p}(\Omega)\setminus \{ 0\}$ a positive minimizer of $\lambda_{1,F}(p,\Omega)$ and denote by $\varphi(v)=|v|^{p-1} v$, we have
\begin{equation}\label{holderNum}
\int_\Omega F(\nabla \varphi(v))\ dx = p \int_\Omega |v|^{p-1}F(\nabla v)\ dx \leq p \left(\int_\Omega |v|^p\ dx \right)^{\frac{p-1}p}\left(\int_\Omega F(\nabla v)^p\ dx\right)^\frac 1p.
\end{equation} 
Let us set $F_t:=\{x\in\Omega \ : \ \varphi(v(x))>t\}$. Since $v\equiv 0$ on $\R^N\setminus\Omega$, then $P_F(F_t;\Omega)=P_F(F_t;\R^n)$ and $F_t\subset \Omega$. Hence, by using the coarea formula \eqref{Fcoarea}, for each $t>0$ we have 
\[
\begin{split}
\int_{\Omega} F(\nabla \varphi(v))\ dx &= \int_{0}^{+\infty}P_F(F_t;\Omega)ds=\int_{0}^{+\infty}P_F(F_t;\R^n)ds\\
& = \int_{0}^{+\infty}\frac{P_F(F_t;\R^n)}{|F_t|}|F_t|ds\geq\inf_{t>0} \frac{P_F(F_t;\R^n)}{|F_t|}\int_{0}^{+\infty}|F_t|ds\geq\\ 
&\geq\inf_{D\subseteq \Omega} \frac{P_F(D;\R^n)}{|D|}\int_{0}^{+\infty}|F_t|ds=h_{1,F}(\Omega)\int_{\Omega}|\varphi(v)|dx\\&=h_{1,F}(\Omega)\int_{\Omega}|v|^pdx
.
\end{split}
\]
We conclude by applying \eqref{holderNum}
\[h_{1,F}(\Omega)\leq p \frac{(\int_{\Omega'} F(\nabla v)^p\ dx )^\frac 1p}{(\int_{\Omega'} |v|^p\ dx )^\frac 1p}=p(\lambda_{1,F}(p,\Omega'))^{\frac 1 p}.
\]
\end{proof}

To give the asymptotic results, we need the following Proposition, inspired to \cite[Corollary 2]{ACV} and \cite[Theorem 5.3]{BoP}.

For any subset $E$ of $\R^n$, we denote the anisotropic $\varepsilon$-strip around $E$ by
\[
E^\varepsilon_F:=\{x\in\R^n \colon d_F(x,E)\leq\varepsilon \}.
\]
\begin{prop}\label{app_per_e}
Let $E\subset\R^n$ be a set with Lipschitz boundary. Then
\[
|E^\varepsilon_F\setminus E| =\varepsilon P_F(E) + o(\varepsilon).
\]
\end{prop}
\begin{proof}
We want to prove that $|E^\varepsilon_F \setminus E |-\varepsilon P_F(E) =o(\varepsilon)$, that is 
\begin{equation*}
\lim_{\varepsilon \to 0}\frac{|E^\varepsilon_F \setminus E|}\varepsilon =P_F(E).
\end{equation*}
Using the coarea formula we have
\begin{equation*}
|E_F^\varepsilon \setminus E|=\int_{E_F^\varepsilon\setminus E}\ dx=\int_{E_F^\varepsilon\setminus E} F(\nabla d_F(x, E))\ dx=\int_0^\varepsilon P_F(\{d_F(x, E)>s\})ds
\end{equation*}
and hence
\begin{equation*}
\begin{split}\lim_{\varepsilon \to 0}\frac{|E^\varepsilon_F\setminus E|}\varepsilon & =\lim_{\varepsilon \to 0}\frac{\ds\int_0^\varepsilon P_F(\{d_F(x,E)>s\})ds}{\ds \varepsilon }=\frac{d}{d\varepsilon}\left[\int_0^\varepsilon P_F(\{d_F(x,E)>s\})ds\right]_{\varepsilon=0}\\
&=\Big[P_F(\{d_F(x, E)>\varepsilon\})\Big]_{\varepsilon=0}=P_F(E).
\end{split}
\end{equation*}
\end{proof}

From \cite[Theorem 4.1]{KN}, we recall the asymptotic behavior of the first eigenvalue of problem \eqref{introauti}, as $p\to 1^+$. 
\begin{thm}
Let $\Omega$ be an open bounded subset of $\R^n$ and $1<p<+\infty$. Then
\[
\lim_{p\to 1^{+}} \lambda_{1,F}(p,\Omega)= h_{1,F}(\Omega).
\]
\end{thm}
\begin{proof}
By the Cheeger inequality \eqref{cheeger_ine}, we only need to show that
\begin{equation*}
\limsup_{p\to 1} \lambda_{1,F}(p,\Omega)\leq h_{1,F}(\Omega).
\end{equation*}
By \Cref{approssimazione_thm_per}, we can approximate the anisotropic Cheeger set $C_1$ of $\Omega$ from the interior, by a sequence of sets of finite anisotropic perimeter $\{C_{m}\}_{m\in\N}$ with smooth boundary converging to $h_{1,F}(\Omega)$. More precisely, we have that
\begin{equation}
\label{PV1m}
\frac{P_F(C_{m})}{|C_{m}|}\leq h_{1,F}(\Omega)+\frac 1m.
\end{equation}

Now, for any fixed $m\in\N$, we consider $\varepsilon>0$ sufficiently small in order to have $C^\varepsilon_{m}\subset\subset C_{1}$. We define $v_m\in W_0^{1,p}$ such that $v_m = 1$ on $C_{m}$, $v_m = 0$ on $\Omega\setminus C_{m}^\varepsilon$ and $F(\nabla v_m) = \frac 1\varepsilon$ on $C_{m}^\varepsilon\setminus C_m$. Hence we have
\[
\lambda_{1,F}(p,\Omega)\leq \frac{\int_\Omega F^{p}(\nabla v_m)\ dx}{\int_\Omega v_m^p\ dx}=\frac{\varepsilon^{-p}| C_{m}^\varepsilon \setminus C_m|}{| C_{m} |}.
\]
Therefore, by applying Proposition \ref{app_per_e} and \eqref{PV1m}, we have
\[
\lambda_{1,F}(p,\Omega)\leq \frac{\varepsilon^{1-p} P_F(C_{m}^\varepsilon)+\varepsilon^{-p} o(\varepsilon)}{| C_{m} |}\leq\frac{1}{\varepsilon^{p-1}} \left( h_{1,F}(\Omega)+\frac 1m\right)+ \frac{o(\varepsilon)}{\varepsilon^{p} }.
\]
We conclude by letting $\varepsilon\to 0^+$ and then $m\to+\infty$.
\end{proof}

From \cite[Theorem 4.2]{KN}, we recall the asymptotic behaviour of the first eigenfunctions of problem \eqref{introauti}, as $p\to 1^+$. 
\begin{prop}
Let $\Omega$ be an open bounded Lipschitz subset of $\R^n$. The first normalized  eigenfunction $u_{1,p}$ of \eqref{introauti} converges, up to a subsequence, to a nonegative limit function $u_1\in BV(\Omega)$ with $||u_1||_1=1$, as $p\to 1^+$. Moreover, almost all level sets $\Omega_t:=\{u_1>t\}$ of $u_1$ are Cheeger sets. 
\end{prop}

\subsection{The relationships between $\lambda_{2,F}(p,\Omega)$ and $h_{2,F}(\Omega)$}
At this stage, we prove a Cheeger type inequality for the second eigenvalue.
\begin{thm}
\label{second_ci}
Let $\Omega$ be a bounded domain of $\R^n$ and $1<p<+\infty$. Then
\[
\lambda_{2,F}(p,\Omega)\ge \left(\frac{h_{2,F}(\Omega)}{p}\right)^p.
\]
\end{thm}
\begin{proof}
By \Cref{key}, we know there exist two disjoint domains $\Omega_1$, $\Omega_2$ of $\Omega$ such that
\begin{equation*}
\begin{split}
\lambda_{2,F}(p,\Omega)&=\max\left\{\lambda_{1,F}(p,\Omega_1),\lambda_{1,F}(p,\Omega_2)\right\}\ge \\ &\ge\left(\frac{1}{p} \max\left\{h_{1,F}(\Omega_1),h_{1,F}(\Omega_2)\right\}\right)^{p} \ge \left(\frac{h_{2,F}(\Omega)}{p}\right)^{p},
\end{split}
\end{equation*}
where we have used the Cheeger inequality \eqref{cheeger_ine} for $h_{1,F}(\Omega)$ and the characterization of $h_{2,F}$.
\end{proof}

\begin{thm}
\label{convergence_thm}
Let $\Omega\subset \R^{n}$ a Lipschitz bounded open set. Then we have 
\begin{equation}
\label{lim1}
\lim_{p\to 1}\lambda_{2,F}(p,\Omega)= h_{2,F}(\Omega).
\end{equation}
\end{thm}
\begin{proof}
Using the inequality of Proposition \ref{second_ci}, we immediately obtain that 
\begin{equation}
\label{lim2}
\liminf_{p\to 1} \lambda_{2,F}(p,\Omega) \ge h_{2,F}(\Omega).
\end{equation}
On the other hand,  let $\Omega_{1}$ and $\Omega_{2}$ two disjoint sufficiently smooth subsets of $\Omega$. By using \Cref{l2} we have:
\begin{multline*}
\limsup_{p\to 1} \lambda_{2,F}(p,\Omega) \le \limsup_{p\to 1} \max\left\{\lambda_{1,F}(p,\Omega_1),\lambda_{1,F}(p,\Omega_2)\right\} = \\ = 
\max\left\{\limsup_{p\to 1}\lambda_{1,F}(p,\Omega_1),\limsup_{p\to 1}\lambda_{1,F}(p,\Omega_2)\right\}=\max\left\{h_{1,F}(\Omega_1),h_{1,F}(\Omega_2)\right\}.
\end{multline*}
Passing to the infimum on $\Omega_{1},\Omega_{2}$, we get
\[
\limsup_{p\to 1} \lambda_{2,F}(p,\Omega) \le h_{2,F}(\Omega),
\]
that, together with \eqref{lim2}, gives the conclusion.
\end{proof}

Another property of $h_{2}$ is the following.
\begin{prop}
Let $\Omega$ be a bounded domain of $\R^{n}$ and let  $u_{2,p}$ be a second eigenfunction associated to $\lambda_{2,F}(p,\Omega)$. Let us denote by $\Omega_1^p$, $\Omega_2^p$ two nodal domains. Then
\begin{equation*}
\lim_{p\to 1} \max\left\{h_{1,F}(\Omega_1^p), h_{1,F}(\Omega_2^p)\right\}=h_{2,F}(\Omega)
\end{equation*}
\end{prop}
\begin{proof}
By definition of $h_{2,F}(\Omega)$, we have that
\begin{equation*}
h_{2,F}(\Omega)\le\max\left\{h_{1,F}(\Omega_1^p),h_{1,F}(\Omega_2^p)\right\}.
\end{equation*}
Hence, it remains to show the reverse limit inequality. We will prove that for all $\varepsilon>0$, there exists $p_0>1$ such that for every $1<p<p_0$
\[
\max\{h_{1,F}(\Omega_1^p),h_{1,F}(\Omega_2^p)\}\leq h_{2,F}(\Omega)+\varepsilon.
\]
By contradiction, we suppose that there exists $\varepsilon >0$ such that, without loss of generality, $h_{1,F}(\Omega_1^{p_k})>h_{2,F}(\Omega)+\varepsilon$, for a subsequence $p_k\to 1$. Then, by the Cheeger's inequality \eqref{cheeger_ine}, we have
\[
\lambda_{2,F}(p_k,\Omega)\geq\left(\frac{h_{1,F}(\Omega_1^{p_k})}{p_k}\right)^{p_k}>\left(\frac{h_{2,F}(\Omega) +\varepsilon}{p_k}\right)^{p_k}>h_{2,F}(\Omega)+\frac\varepsilon 2
\]
for $k$ large enough. But this is in contradiction with $\limsup_{p\to 1}\lambda_{2,F}(p,\Omega)=h_{2,F}(\Omega)$.
\end{proof}
This Proposition implies the following corollary result.
\begin{prop}
\label{cor_mis_set}
The measures of $\Omega_1^p$ and $\Omega_2^p$ are uniformly bounded from below, for $p\to 1$, by $\kappa_n(n/2h_{2,F}(\Omega))^n$.
\end{prop}
\begin{proof}
By the previous proposition, there exists $p_0>1$ such that, for every $1<p<p_0$
\[ 2 h_{2,F}(\Omega)\ge
\max\{h_{1,F}(\Omega_1^p),h_{1,F}( \Omega_2^p)\}.
\]
Then by using \eqref{faberkrahncheeger}, we get 
\begin{equation*}
\left(\frac{2 h_{2}(\Omega)}{ n \kappa_{n}^{\frac1n}} \right)^{{n}}\ge  \max\left\{\frac{1}{|\Omega_{1}^{p}|},\frac{1}{|\Omega_{2}^{p}|}  \right\}.
\end{equation*}

\end{proof}

Now, to investigate the asymptotic behavior of the second eigenfunctions as $p\to 1$, we start by giving three technical lemmas.
\begin{lem}
\label{vartotest}
Let $\Omega\subset\R^n$ be a bounded domain with Lipschitz boundary, $\{p_j\}_{j\in\N}$ be a sequence of real numbers $p_j\geq 1$ such that $p_j\to 1^+$ as $j\to+\infty$ and $\{u_j\}_{j\in\N}\subseteq W_0^{1,p_j}(\Omega)$ be a sequence such that $u_j\to u$ in $L^1(\Omega)$ as $j\to+\infty$. Then
\[
\int_\Omega |Du|_F\leq\liminf_{j\to+\infty}\int_\Omega F(\nabla u_j)^{p_j}\ dx
\]
\end{lem}
\begin{proof}
The proof follows line by line the one contained in {\cite[Lemma 5.8]{Pa1}}. For sake of completeness, we write it in details. Since $\partial\Omega$ is Lipschitz, the function $u_j$ are in particular in $BV(\Omega)$. We denote by $p'_j$ the conjugate exponent to $p_j$. Then $p'_{j}\to +\infty$ as $j\to +\infty$. By the lower semicontinuity of the anisotropic total variation, H\"older inequality and Young inequality, we have
\begin{equation*}
\begin{split}
\int_\Omega |Du|_F & \leq\liminf_{j\to+\infty}\int_\Omega |Du_j|_F=\liminf_{j\to+ \infty}\int_\Omega F(\nabla u_j) \ dx\\
& \leq \liminf_{j\to+\infty}\left(\int_\Omega F^{p_j}(\nabla u_j) \ dx\right)^\frac 1 {p_j} |\Omega|^\frac 1 {p'_j}\\
& \leq\liminf_{j\to+\infty}\left[\int_\Omega F^{p_j}(\nabla u_j)\ dx +|\Omega|\frac{p_j^{-p'_j/p_j}}{p'_j}\right]\\&\leq\liminf_{j_\to+\infty}\int_\Omega F^{p_j}(\nabla u_j)\ dx+\limsup_{j\to\infty}|\Omega|\frac{p_j^{-\frac{p_j'}{p_j}}}{p_j'}\\&=\liminf_{j\to+\infty}\int_\Omega F^{p_j}(\nabla u_j)\ dx.
\end{split}
\end{equation*}
\end{proof}
\begin{lem}{\cite[Lem. 5.9]{Pa1}.}
\label{limitu}
Let $\Omega\subset\R^n$ be an open bounded set, $\{p_j\}_{j\in\N}$ be a sequence of real numbers $p_j\geq 1^+$ such that $p_j\to 1^+$  as $j\to +\infty$ and $\{u_j\}_{j\in\N}\subseteq L^\infty(\Omega)$ be a sequence such that $0<||u_j||_{L^\infty(\Omega)}<c$ for a positive constant $c$ and for every $j\in\N$. Then
\[
\lim_{j\to+\infty}\int_\Omega |u_j|^{p_j}\ dx = \int_\Omega |u|\ dx
\]
\end{lem}
\begin{lem}{\cite{DPGP2}, \cite[Lem. 5.1]{LU}, \cite{AFT}.}
 \label{norm_estimate}
Let $\Omega$ be an open bounded subset of $\R^n$ and $u_{2,p}$ be a second eigenfunction of \eqref{introauti}. Then
 \begin{equation*}
 \|u_{2,p}\|_{L^\infty(\Omega)} \le C_{n,p,F}\lambda_{2,F}(p,\Omega)^\frac{n}{p}\|u_{2,p}\|_{L^{1}(\Omega)},
 \end{equation*} 
 where $C_{n,p,F}$ is a constant depending only on $n$, $p$ and $F$.
\end{lem}
\begin{thm}
Let $\Omega\subset \R^{n}$ a Lipschitz bounded open set and $u_{2,p}$ be a second eigenfunction of \eqref{introauti} such that $\|u_{2,p}\|_{L^p(\Omega)}=1$. Then, up to a subsequence, $u_{2,p}$ converges in $L^1(\Omega)$, as $p\to 1^+$, to a function $u_2\in BV(\Omega)$ such that $\|u_2\|_{L^1(\Omega)}=1$ and $\int_\Omega |D u_2|_F\le h_{2,F}(\Omega)$. Moreover, $u_2$ cannot be strictly positive or strictly negative.
\end{thm}
\begin{proof}
By \Cref{convergence_thm}, \Cref{norm_estimate} and the H\"older inequality, we have that $u_{2,p}$ are uniformly bounded in $L^\infty(\Omega)$. Moreover, we have
\[
\int_\Omega|Du_{2,p}|_F=\int_{\Omega}F(\nabla u_{2,p})\ dx \leq \left(\int_\Omega F(\nabla u_{2,p})^p\ dx\right)^\frac 1 p |\Omega|^\frac 1 {p'}=\lambda_{2,F}(p,\Omega)^\frac 1 p |\Omega|^\frac 1 {p'},
\]
where $p'$ is the conjugate exponent to $p$. Since $\lambda_{2,F}(p, \Omega)\to h_{2,F}(\Omega)$, the functions are uniformly bounded in $BV(\Omega)$, hence there exists a subsequence converging in $L^1(\Omega)$ to a function $u_2\in BV(\Omega)$. Then, \Cref{vartotest} implies
\[
\int_\Omega |Du_2|_F\leq \liminf_{p\to 1 }\lambda_{2,F}(p,\Omega)=h_{2,F}(\Omega).
\]
and \Cref{limitu} yields $||u_2||_{L^1(\Omega)}=1$. Finally, by \Cref{cor_mis_set}, we have that $u_2$ cannot be strictly positive or strictly negative.
\end{proof}

\section{A constrained anisotropic Cheeger constant}\label{twisted_sec}
Throughout this Section, we set $BV_0(\Omega):=\{ u\in BV(\R^n): u\equiv 0 \ \text{in}\ \R^n\backslash\Omega \}$ and we examine a constrained Cheeger type problem:
\begin{equation}
\label{twisted_cheeger}
\mathcal K_{q,F} (\Omega):=\min\left\{ \frac{\ds\int_{\R^{n}} |Du|_F }{\left(\ds\int_\Omega |u|^{q} \ dx\right)^{\frac1q}}; \ u\in BV_0(\Omega), \ u\not\equiv 0, \ \int_\Omega u\ dx=0\right\},
\end{equation}
where $1\le q <\frac{n}{n-1}$ and $\Omega$ is a bounded open set. From a different point of view, we can consider problem \eqref{twisted_cheeger} as the relaxed form of the following problem:
\[
\mathcal K_{q,F} (\Omega):=\inf\left\{ \frac{\ds\int_\Omega F(\nabla u) \ dx}{\left(\ds\int_\Omega |u|^{q} dx\right)^{\frac1q}}; \ u\in W_0^{1,1}(\Omega), \ u\not\equiv 0, \ \int_\Omega u\ dx=0\right\}
\]
and in this case we can refer to $\mathcal K_{q,F}(\Omega)$ as the \lq\lq twisted\rq\rq\ anisotropic Cheeger constant, that is the $BV$-counterpart of the twisted eigenvalue for the anisotropic Laplacian (see \cite{FH} and \cite{Pi1}):
\[
\lambda^T(\Omega)=\min\left\{ \frac{\ds\int_\Omega F(\nabla u)^{2}  dx}{\ds\int_\Omega |u|^{2} \ dx}; \ u\in H_0^1(\Omega), \ u\not\equiv 0, \ \int_\Omega u dx=0 \right\}.
\]
For other details on  ``twisted''  eigenvalues we refer to \cite{BB, FH}.

We stress that
\[
h_{1,F}(\Omega)\le \mathcal K_{1,F}(\Omega)\le  h_{2,F} (\Omega).
\]

Here we aim to study the minimization problem
\begin{equation*}
\min_{|\Omega|=c} \mathcal K_{q,F}(\Omega)
\end{equation*}
for $c>0$ fixed. Following the ideas of \cite{BDNT}, we have:
\begin{thm}
Let $\Omega$ be an open bounded set of $\R^n$, $n\geq 2$, and $1\le q<\frac{n}{n-1}$, then 
\[
\mathcal K_{q,F}(\Omega)=\min\left\{\frac {\frac{P_F(G_1)}{|G_1|}+\frac{P_F(G_2)}{|G_2|}}{(|G_1|^{1-q}+|G_2|^{1-q})^\frac 1 q},\ G_i\subseteq\Omega,  0<|G_i| ,\ G_1 \cap G_2 = \emptyset\right\}=:\ell(\Omega).
\]
\end{thm}
\begin{proof}

First of all, proceeding exactly as in \cite{BDNT}, it is possible to show that $\ell(\Omega)$ is actually a minimum. 
We set 
\[
J_{q}(u):=\frac{\ds\int_{\R^{n}} |Du|_F }{\left(\ds\int_\Omega |u|^{q} dx\right)^{\frac1q}}
\]
Let $G_1,G_2$ be  two disjoint  subsets of $\Omega$, and consider $U=|G_{1}|\chi_{G_{2}}-|G_{2}|\chi_{G_{1}}$ as test function in $J_{q}$. We have
that $U$ is an admissible test function for $\mathcal K_{q,F}( \Omega)$, and
\[
J_{q}(U)=\frac {\frac{P_F(G_1)}{|G_1|}+\frac{P_F(G_2)}{|G_2|}}{(|G_1|^{1-q}+|G_2|^{1-q})^\frac 1 q}.
\]

Then we immediately have that
\[
\mathcal K_{q,F}( \Omega)\le\ell (\Omega).
\]
To prove the opposite inequality, we have to show that for any admissible $u\in BV_0(\Omega)$, we have
\[
\mathcal K_{q,F}( \Omega)\geq\ell (\Omega).
\]
We set $\mu_\pm(t):=|\{u_\pm>t\}|$, $p_\pm(t):=P_F(\{u_\pm>t\})$ and $\Omega_\pm:=\spt u_\pm$, where $u_+$ and $u_-$ are, respectively, the positive and the negative part of $u$. Being $\int_{\Omega}u dx=0$, then if $u\not \equiv 0$, it holds $|\Omega_{\pm}|>0$ and $\int_{\Omega^+} u_+\, dx=\int_{\Omega^-} u_-\, dx:=M$.  By the Fubini Theorem and the H\"older inequality, it is not difficult to prove that 
\begin{equation}
  \label{eq:1}
\left(\int_\Omega {|u|}^q dx\right)^{\frac 1 q} \le \left[\left(
  \int_0^{\esssup u_+} \mu_+(t)^{\frac 1 q} dt \right)^q+ \left(
  \int_0^{\esssup u_-} \mu_-(s)^{\frac 1 q} ds \right)^q \right]^{\frac 1
q}.
\end{equation}
Now we perform the change of variables
\[
\xi(t) = \int_0^t \mu_+(\sigma) d\sigma,\quad t \in [0,\esssup u^+]
, \qquad
\eta(s) = \int_0^s \mu_-(\tau) d\tau,\quad s \in[0,\esssup u^-],
\]
respectively, in both integrals in the right-hand side of
\eqref{eq:1}. The functions $\xi$ and $\eta$ are strictly increasing and we 
have that $\xi(t)\le M=\xi(\esssup u^+)$ and $\eta(s)\le 
M=\eta(\esssup u^-)$. Hence, from 
\eqref{eq:1} and the Minkowski inequality it follows that 
\begin{equation}
  \label{eq:2}
\left(\int_\Omega {|u|}^q dx\right)^{\frac 1 q} \le 
  \int_0^{M}
  \left[
    \mu_+(t(r))^{{1-q}} + \mu_-(s(r))^{{1-q}}\right]^{\frac 1 q}
  dr.
\end{equation}
On the other hand,
denoting by $p_\pm(t)=P_{F}(\{u_\pm>t\})$ for $t\ge0$, the co-area formula for BV functions yields 
\begin{equation}
  \label{eq:3}
 \int_{\R^{n}} |Du|_F = \int_0^{+\infty} p_+(t)\,dt + \int_0^{+\infty} p_-(s)\, ds = 
  \int_0^M \left[
    \frac {p_+(t(r))}{\mu_+(t(r))} + \frac{p_-(s(r))}{\mu_-(s(r))}
    \right]dr,
\end{equation}
where we performed the change of variables $t=t(\xi)$, $s=s(\eta)$ defined
above. Finally, combining \eqref{eq:2} and \eqref{eq:3} we have
\begin{multline*}
J_q(u)\ge \frac{\int_0^M \big[
    \frac {p_+(t(r))}{\mu_+(t(r))} + \frac{p_-(s(r))}{\mu_-(s(r))}
    \big]dr}{ \int_0^{M}
  \big[
    \mu_+(t(r))^{{1-q}} + \mu_-(s(r))^{{1-q}}\big]^{\frac 1 q}
    dr } \ge
    \\[.2cm] \ge \inf_{0<r<M} \frac{  \frac {p_+(t(r))}{\mu_+(t(r))} +
    \frac{p_-(s(r))}{\mu_-(s(r))} }{\big[ \mu_+(t(r))^{{1-q}} +
    \mu_-(s(r))^{{1-q}}\big]^{\frac 1 q} } 
  \ge
  \ell(\Omega),
\end{multline*}
and this concludes the proof.\end{proof}
Finally we give the last main result.
\begin{thm}
Let $\Omega$ be an open bounded set in $\R^{n}$, and $1\le q<\frac{n}{n-1}$. Every minimizer of $\mathcal K_q(\Omega)$ in the class of bounded open sets with given measure, is union of two disjoint Wulff shapes. Then there exists $\tilde q=\tilde q(n)\in]1,\frac n {n-1}[$ such that, when $q<\tilde q$, the minimizer is unique and it is the union of two disjoint Wulff shape with the same radius and when $q>\tilde q$, the minimizer is unique and it is the union of two disjoint Wulff shape with different radii.
In addiction
\begin{itemize}
\item If $n=2$, then $\tilde q =\frac 74$, the minimizer is unique even at $q=\tilde q$ and consisting in the union of two disjoint Wulff shapes of equal radii.
\item For every $n\geq 3$, the minimizer is not unique at $q=\tilde q$, indeed there are exactly two minimizers, one of which is the union of two disjoint balls with equal radii.
\end{itemize}
\end{thm}

\begin{proof}
The previous result gives that
\[
\mathcal K_{q,F}(\Omega)=\min\left\{\frac {\frac{P_F(G_1)}{|G_1|}+\frac{P_F(G_2)}{|G_2|}}{(|G_1|^{1-q}+|G_2|^{1-q})^\frac 1 q},\ G_1, G_2\subseteq\Omega,  0<|G_1|,|G_2| ,\ G_1 \cap G_2 = \emptyset\right\}.
\]
The first step in order to prove the theorem is to reduce the minimization problem to the class of the sets which are union of two disjoint Wulff shapes $\mathcal W_{1}$ and $\mathcal W_{2}$ such that $|\mathcal W_{1}|+|\mathcal W_{2}|=|\Omega|$. This follows by using the anisotropic isoperimetric inequality \eqref{isoine}, and then reasoning  as in \cite{BDNT}. After that, the problem is reduced to a one-dimensional minimization problem which is exactly the same studied in \cite{BDNT}. The argument therein contained leads to the result.

\end{proof}

\section*{Acknowledgements}
This work has been partially supported by GNAMPA of INdAM. The author wants to thank Professor Bernd Kawohl for the suggestions he gave him during his period in K\"oln.

\end{document}